\documentclass[10pt,a4paper]{article}%
\usepackage{amsmath}
\usepackage{amsfonts}
\usepackage{amssymb}
\usepackage{graphicx}
\usepackage{color}%
\usepackage{amssymb}
\usepackage{graphicx}
\usepackage[compress]{cite}%

\newtheorem{theorem}{Theorem}

\newtheorem{corollary}[theorem]{Corollary}

\newtheorem{definition}[theorem]{Definition}

\newtheorem{lemma}[theorem]{Lemma}

\newtheorem{proposition}[theorem]{Proposition}
\newtheorem{remark}[theorem]{Remark}

\newenvironment{proof}[1][Proof]{\noindent\textbf{#1.} }{\ \rule{0.5em}{0.5em}}

\begin{document}

\title{Partial Deranged Bell Numbers and Their Combinatorial Properties}
\author{Yahia Djemmada$^a$\thanks{yahia.djemmada@nhsm.edu.dz}, Levent Karg\i n$^b$\thanks{lkargin@akdeniz.edu.tr (Corresponding author)} and M\"{u}m\"{u}n Can$^b$\thanks{mcan@akdeniz.edu.tr}
\\{\small $^a$National Higher, School of Mathematics, Mahelma 16093, Sidi Abdellah, Algeria}\\{\small
$^a$RECITS Laboratory BP 32,  USTHB, El Alia 16111 Bab Ezzouar, Algiers,  Algeria}\\{\small
$^b$Department of Mathematics, Akdeniz University, Antalya, 07058, T\"{u}rkiye}}
\date{}
\maketitle

\begin{abstract}
We introduce a novel generalization of deranged Bell numbers by defining the
partial deranged Bell numbers $w_{n,r}$, which count the number of set
partitions of $\left[  n\right]  $ with exactly $r$ fixed blocks, while the
remaining blocks are deranged. This construction provides a unified framework
that connects partial derangements, Stirling numbers, and ordered Bell
numbers. We investigate their combinatorial properties, including explicit
formulas, generating functions, and recurrence relations. Moreover, we
demonstrate that these numbers are expressible in terms of classical sequences
such as deranged Bell numbers and ordered Bell numbers, and reveal their
relationship to complementary Bell numbers, offering insights relevant to
Wilf's conjecture. Notably, we derive the identity
\[
\tilde{\phi}_{n}=\Tilde{w}_{n,0}-\Tilde{w}_{n,1}=\tilde{w}_{n-1,0}-2\tilde
{w}_{n-1,2},
\]
which illustrates their structural connection to complementary Bell numbers.
We also introduce a polynomial expansion for these numbers and explore their
connections with exponential polynomials, geometric polynomials, and Bernoulli
numbers. These relationships facilitate the derivation of closed-form
expressions for certain finite summations involving Stirling numbers of the
second kind, Bernoulli numbers, and binomial coefficients, articulated through
partial derangement numbers.


\end{abstract}

\section{Introduction}

\subsection{Permutations with fixed points}

In combinatorics, a permutation $\sigma$ of an $n$-element set $S_{n}%
=\{a_{1},a_{2},\dots,a_{n}\}$ is a bijective map $\sigma:S_{n}\rightarrow
S_{n}$; or simply, it is an ordering of the elements of $S_{n}$, represented
by
\[
\sigma(S_{n})=(\sigma(a_{1})\,\sigma(a_{2})\,\cdots\,\sigma(a_{n})).
\]
It is easy to show that the total number of permutations of an $n$-set is $n!$.

Within permutations, the notion of a fixed point is defined as an element
$a_{i}$ satisfying $\sigma(a_{i}) = a_{i}$. For example, the permutation
$(132465)$ of the set $\{1,2,3,4,5,6\}$ has two fixed points: $1$ and $4$.

The concept of derangements (or permutations with no fixed points) emerged
from the \emph{probl\`{e}me des rencontres}, introduced by Pierre R\'{e}mond
de Montmort in the 1708 edition of his treatise \cite{Mon}. A complete
solution to this problem, formulated through collaboration with Nicholas
Bernoulli, later appeared in the second expanded edition of the treatise,
published in 1713 \cite{Mon}. We can state the problem as finding the number
of derangements of an $n$-set, denoted $d_{n}$; formally, we can write
\[
d_{n}=\left\vert \{\sigma\in S_{n}\mid\sigma(a_{i})\neq a_{i}\text{ for all
}i\in\{1,2,\dots,n\}\}\right\vert .
\]
The enumeration of derangements has been explored since then, with the
following formula derived using the inclusion-exclusion principle:
\[
d_{n}=n!\sum_{i=0}^{n}\frac{(-1)^{i}}{i!},\quad n\geq0.
\]
A natural generalization of the derangement problem is to determine the number
of permutations with exactly $r$ fixed points. This extension is known as the
generalized rencontre problem or the partial derangements problem \cite{Rio}.
The number of such permutations, denoted $d_{n,r}$, is given by
\[
d_{n,r}=\binom{n}{r}d_{n-r}=\frac{n!}{r!}\sum_{i=0}^{n-r}\frac{(-1)^{i}}%
{i!},\quad n\geq r\geq0.
\]
For $r=0$, we recover the derangement numbers. Moreover, the exponential
generating function for the sequence $\{d_{n,r}\}_{n\geq r}$ is given by
\begin{equation}
\sum_{n\geq r}d_{n,r}\frac{t^{n}}{n!}=\frac{t^{r}}{r!}\frac{e^{-t}}%
{1-t}.\label{dkrgf}%
\end{equation}

The studies \cite{MF,WMM} can be consulted for other approaches related to
derangement numbers.

\subsection{Partitions of a set}

To simplify notation, we consider the set $S_{n}$ to be the standard set of
the first $n$ positive integers, $[n] = \{1,2,\dots,n\}$.

The \emph{Stirling numbers of the second kind}, denoted by ${n\brace k}$,
count the number of ways to partition a set of $n$ elements $[n]$ into $k$
nonempty subsets. They satisfy the following explicit formula \cite{QG}
\begin{equation}
{n\brace k}=\frac{1}{k!}\sum_{i=0}^{k}(-1)^{k-i}{\binom{k}{i}}i^{n}.\label{2}%
\end{equation}
It is obvious that ${0\brace0}=1$, ${n\brace0}=0$ for $n>0$, and ${n\brace
k}=0$ for $k>n$. The Bell numbers $\phi_{n}$, defined by \cite{QG}%
\[
\phi_{n}=\sum_{k=0}^{n}{n\brace k},
\]
count the total number of ways to partition a set of $n$ distinguishable
elements into nonempty, disjoint subsets, where neither the order of elements
within subsets nor the order of the subsets themselves matters (see also
\cite{B,Dil2,FR}). Moreover, the complementary Bell numbers (also known as Rao
Uppuluri-Carpenter numbers), defined by%
\[
\tilde{\phi}_{n}=\sum_{k=0}^{n}\left(  -1\right)  ^{k}{n\brace k},
\]
record the difference between the number of partitions of $[n]$ with an even
number of blocks and those with an odd number of blocks. There is a remarkable
conjecture attributed to Wilf in the context of the theory of partitions for
these numbers and Stirling numbers of the second kind. The conjecture states
that
\[
\text{For all }n\neq2,\text{ }\tilde{\phi}_{n}\neq0.
\]
For more details on Wilf's conjecture, see reference \cite{AM} and the
references therein.

Recall that, as defined above, a permutation is an ordering of the elements of
a set. An \emph{ordered partition} of $[n]$ can be viewed as first
partitioning the set into $k$ nonempty, disjoint blocks and then ordering
these blocks. A common convention is to arrange the blocks in increasing order
of their smallest elements, that is, if we write the blocks as
\[
B_{1}\mid B_{2}\mid\cdots\mid B_{k},
\]
we require
\[
\min(B_{1})<\min(B_{2})<\cdots<\min(B_{k}).
\]
This arrangement is unique; however, any permutation of these $k$ blocks
yields an ordered partition. Thus, partitioning $[n]$ into $k$ blocks can be
done in ${n\brace k}$ ways, and ordering the $k$ blocks can be achieved in
$k!$ ways. Consequently, the total number of ordered partitions, known as the
\emph{ordered Bell numbers} and denoted by $w_{n}$, is given by \cite{T}%
\[
w_{n}=\sum_{k=0}^{n} {n \brace k} k!,\quad n\geq0.
\]
Alternatively, the ordered Bell numbers can be characterized by their
exponential generating function given by
\begin{equation}
\sum_{n=0}^{\infty}w_{n}\frac{t^{n}}{n!}=\frac{1}{2-e^{t}}. \label{egOB}%
\end{equation}

Numerous generalizations of the Bell and ordered Bell numbers have been developed in connection with various generalizations of the Stirling numbers. See, for instance, references \cite{AN,NBCC}, as well as the references therein.
\subsection{Motivation}

In the recent paper on \emph{deranged Bell numbers} \cite{BDL}, the
combinatorial interpretation of the ordered Bell numbers, which count all
ordered partitions of $[n]$, is explored. In that context, the concept of a
deranged partition, as a free-fixed-block permutation of its blocks, is
introduced, and the deranged Bell numbers are defined as the total number of
such partitions of $[n]$. The deranged Bell numbers, denoted by $\Tilde{w}%
_{n}$, can be expressed in terms of the Stirling numbers of the second kind
and the derangement numbers as follows:
\[
\Tilde{w}_{n}=\sum_{k=0}^{n} {n \brace k}d_{k},\quad n\geq0.
\]
Their exponential generating function is given by:
\[
\sum_{n=0}^{\infty}\Tilde{w}_{n}\frac{t^{n}}{n!}=\frac{e^{-(e^{t}-1)}}%
{2-e^{t}}.
\]

Inspired by the generalization of derangements to partial derangements, we
extend the deranged Bell numbers by introducing \emph{partial deranged Bell
numbers}, denoted by $\Tilde{w}_{n,r}$. These numbers count partitions of
$[n]$ in which exactly $r$ blocks remain fixed (i.e., are not permuted) while
the remaining blocks are deranged, or permuted without fixed positions. This
provides a natural generalization of both partial derangements and deranged
Bell numbers, further bridging the relationship between partitioning and
permutation statistics. We explore their combinatorial properties, including
generating functions, explicit formulas, and other related identities.
Moreover, we identify that partial deranged Bell numbers can be represented in
terms of deranged Bell numbers (Theorem \ref{teo1}), and ordered Bell numbers
(Theorem \ref{teo2}). This generalization also provides a novel perspective on
Wilf's conjecture, as we discover that these numbers are closely related to
the complementary Bell numbers (Theorem \ref{Prop:difw}). In particular, we
obtain the identities%
\[
\tilde{\phi}_{n}=\Tilde{w}_{n,0}-\Tilde{w}_{n,1}=\tilde{w}_{n-1,0}-2\tilde
{w}_{n-1,2}.
\]
This result suggests that these numbers may exhibit a close and natural
connection to the complementary Bell numbers, potentially contributing to a
deeper understanding of the structural properties of partition-based
sequences. Finally, following the classical approach used in the polynomial
expansions of Bell numbers and ordered Bell numbers, we introduce a polynomial
expansion for the partial deranged Bell numbers and examine several of their
fundamental properties. In particular, we explore their connections with
exponential polynomials, geometric polynomials, and Bernoulli numbers. These
relationships enable us to derive explicit formulas for certain finite sums
involving Stirling numbers of the second kind, Bernoulli numbers, and binomial
coefficients, expressed in terms of partial derangement numbers (See Section
\ref{Sec3}). These results not only enrich the combinatorial framework
surrounding these sequences but also open up new avenues for future investigation.

\section{The partial deranged Bell numbers}

\begin{definition}
An \emph{$r$-partial deranged partition} $\tilde{\pi}$ of $[n]$ is a
permutation $\sigma$ of the blocks $B_{1}, B_{2}, \dots, B_{k}$ of a partition
of $[n]$ such that exactly $r$ blocks remain fixed. That is, if
\[
\tilde{\pi}([n]) := \sigma(\pi([n])) = \sigma(B_{1}) \mid\sigma(B_{2})
\mid\cdots\mid\sigma(B_{k}),
\]
then $\sigma(B_{i})=B_{i}$ for exactly $r$ indices $i$.
\end{definition}

Below is an example for a partition of $[5] = \{1,2,3,4,5\}$ into 3 blocks and
the corresponding permutations of its blocks with $0$, $1$, $2$, and $3$ fixed blocks.

Consider the partition
\[
\{1,2\}\mid\{3,4\}\mid\{5\}.
\]
Label the blocks as $B_{1}=\{1,2\}$, $B_{2}=\{3,4\}$, and $B_{3}=\{5\}$. Then
the possible permutations of these blocks, classified by the number of fixed
blocks, are:%
\[%
\begin{array}
[c]{|c|c|c|c|c|}\hline
\mathbf{Partition} & 0\text{-fixed} & 1\text{-fixed} & 2\text{-fixed} &
3\text{-fixed}\\\hline
B_{1}\mid B_{2}\mid B_{3} &
\begin{array}
[c]{c}%
B_{2}\mid B_{3}\mid B_{1}\\[1mm]%
B_{3}\mid B_{1}\mid B_{2}%
\end{array}
&
\begin{array}
[c]{c}%
B_{1}\mid B_{3}\mid B_{2}\\[1mm]%
B_{2}\mid B_{1}\mid B_{3}\\[1mm]%
B_{3}\mid B_{2}\mid B_{1}%
\end{array}
& \text{None} & B_{1}\mid B_{2}\mid B_{3}\\\hline
\end{array}
\]
Let the $r$-Partial Deranged Bell numbers (abbreviated as $r$-PDB numbers)
denote the total number of $r$-partial deranged partitions of $[n]$.

\begin{proposition}
\label{Prop:PDB} For $n\geq0$, the $n$$th$ $r$-PDB number, denoted $\Tilde
{w}_{n,r}$, is given by
\[
\Tilde{w}_{n,r}=\sum_{k=0}^{n}\,{n\brace k}d_{k,r}.
\]

\end{proposition}

\begin{proof}
By double counting, we first note that there are ${n \brace k}$ ways to
partition $[n]$ into $k$ blocks. For each such partition, the blocks can be
permuted so that exactly $r$ of them remain fixed in $d_{k,r}$ ways. Summing
over all possible $k$ yields the stated result.
\end{proof}

\begin{theorem}
\label{teo1}For any nonnegative integers $n$ and $r$, the following identity
holds:%
\[
\Tilde{w}_{n,r}=\sum_{k=r}^{n}{\binom{n}{k}}{k\brace r}\Tilde{w}_{n-k}.
\]

\end{theorem}

\begin{proof}
We prove the identity combinatorially. Consider ordered partitions of an
$n$-element set with $r$ fixed blocks. Construct such partitions by:

\begin{itemize}
\item Choosing $k$ elements from $n$ elements ($\binom{n}{k}$ ways).

\item Partitioning the chosen $k$ elements into $r$ fixed blocks (${k
\brace r}$ ways).

\item Partitioning the remaining $n-k$ elements into deranged blocks
($\Tilde{w}_{n-k}$ ways).
\end{itemize}

Summing over all valid $k$ ($0 \leq k \leq n$) yields:
\[
\sum_{k=0}^{n} \binom{n}{k} {k \brace r} \Tilde{w}_{n-k}%
\]
Therefore, the identity holds.
\end{proof}

\begin{theorem}
The number of all deranged partitions $\Tilde{w}_{n}$, can be expressed as:%

\[
\Tilde{w}_{n}=\sum_{r=0}^{n}\frac{(-1)^{r}}{r!}\text{ }_{r}w_{n},
\]
where $_{r}w_{n}$ is defined by%
\[
_{r}w_{n}=\sum_{k=r}^{n}{n\brace k}k!,
\]
and may be called \emph{truncated ordered Bell numbers.}
\end{theorem}

\begin{proof}
Let $S$ be the set of all ordered partitions of an $n$-element set. For each
$1 \leq i \leq k$, let $A_{i}$ denote the subset of partitions in $S$ where
the $i$-th block is fixed. We aim to determine the number of partitions in $S$
where no block is fixed, i.e., $|\bigcap_{i=1}^{k} \overline{A_{i}}|$. By the
principle of inclusion-exclusion, we have:
\[
|\bigcap_{i=1}^{k} \overline{A_{i}}| = |S| - \sum_{r=1}^{n} (-1)^{r-1} \sum_{1
\leq i_{1} < i_{2} < \dots< i_{r} \leq k} |A_{i_{1}} \cap A_{i_{2}} \cap
\dots\cap A_{i_{r}}|.
\]
The term $|S|$ represents the total number of ordered partitions of $n$
elements into $k$ blocks, given by $k! {n \brace k}$. The term $\sum_{1 \leq
i_{1} < i_{2} < \dots< i_{r} \leq k} |A_{i_{1}} \cap A_{i_{2}} \cap\dots\cap
A_{i_{r}}|$ counts the number of partitions with $r$ specific blocks fixed. To
construct such a partition:

\begin{enumerate}
\item Partition the $n$ elements into $k$ blocks in ${n \brace k}$ ways.

\item Choose $r$ blocks from the $k$ blocks to fix, which can be done in
$\binom{k}{r}$ ways.

\item Permute the remaining $(k-r)$ blocks in $(k-r)!$ ways.
\end{enumerate}

Therefore, the number of partitions with $r$ specific blocks fixed is $(k-r)!
{n \brace k} \binom{k}{r} = \frac{k!}{r!} {n \brace k}$. Substituting this
into the inclusion-exclusion formula and simplifying, and summing over all
possible values of $k$ from 0 to $n$, we obtain:
\[
\Tilde{w}_{n} = \sum_{r=0}^{n} \frac{(-1)^{r}}{r!} \sum_{k=r}^{n} k! {n
\brace k}.
\]
This completes the proof.
\end{proof}


Among the various analytical tools employed in the study of combinatorial
number families, the Mellin derivative operator, defined by (for any
appropriate function $f$)%
\[
(xD)f(x):=x\frac{d}{dx}f(x),
\]
stands out as one of the actively used methods. This operator has also been
used as a tool in Euler's work and plays a significant role in establishing
recurrence relations, and uncovering deep structural identities within
combinatorial sequences. Its applications extend to various families such as
Stirling numbers, Bell numbers, and their generalizations, offering a powerful
framework for both symbolic and asymptotic analysis (see, for example,
\cite{B,BoyadzhievandDil,Dil3,Dil1,Kargin3,Kargin4,K}). The key identity for
the Mellin derivative acting on a smooth function $f(x)$ is given by (cf.
\cite[p. 89]{Gould})
\begin{equation}
(xD)^{n}f(x)=\sum_{k=0}^{n}{n\brace k}\,x^{k}f^{(k)}(x),\label{eq:xd_trans}%
\end{equation}
where $(xD)^{n}$ denote the $n$-fold application of this operator. These are
the most well-known numbers that occur when applying Mellin derivative to the
appropriate function \cite{B}:%
\begin{align}
\left.  (xD)^{n}\left(  e^{x}\right)  \right\vert _{x=1} &  =e\phi
_{n},\label{4}\\
\left.  (xD)^{n}\left(  e^{-x}\right)  \right\vert _{x=1} &  =e^{-1}%
\tilde{\phi}_{n},\label{8}\\
\left.  (xD)^{n}\left(  \frac{1}{2-x}\right)  \right\vert _{x=1} &
=w_{n},\label{7}\\
\left.  (xD)^{n}\left(  \left(  x-1\right)  ^{k}\right)  \right\vert _{x=1} &
=k!{n\brace k}\,.\text{ }\label{3}%
\end{align}
Therefore, when the Mellin derivative is applied, the question of which
function generates these numbers naturally becomes a subject of interest, and
this is addressed in the lemma below.

\begin{lemma}
\label{lem:partial_deranged_bell} For all integers $n$ and nonnegative
integers $r$, the $r$-PDB numbers $\Tilde{w}_{n,r}$ are given by
\[
\Tilde{w}_{n,r}=\frac{e}{r!}\left.  (xD)^{n}\left(  \frac{(x-1)^{r}e^{-x}%
}{2-x}\right)  \right\vert _{x=1}.
\]

\end{lemma}

\begin{proof}
Utilizing (\ref{dkrgf}), the generating function (for fixed $k$) of the
Stirling numbers of the second kind \cite{QG}
\begin{equation}
\sum_{n=k}^{\infty}{n\brace k}\frac{t^{n}}{n!}=\frac{(e^{t}-1)^{k}}{k!}.
\label{s2gf}%
\end{equation}
and Proposition \ref{Prop:PDB}, we first derive the exponential generating
function for the $r$-PDB numbers as follows:%
\begin{equation}
G(t;r):=\sum_{n\geq0}\Tilde{w}_{n,r}\frac{t^{n}}{n!}=\frac{(e^{t}%
-1)^{r}e^{-(e^{t}-1)}}{r!(2-e^{t})}. \label{tegf}%
\end{equation}
By substituting $x=e^{t}$, the function can be rewritten in terms of $f(x)$
as:
\[
f(x)=\frac{e}{r!}\frac{(x-1)^{r}e^{-x}}{(2-x)}.
\]
Taking derivatives with respect to $t$, we naturally encounter the
differential operator $(xD)$ which leads to the identity:%
\[
\frac{d}{dt}f(e^{t})=(xD)f(x)\text{, and by induction }\left(  \frac{d}%
{dt}\right)  ^{n}f(e^{t})=(xD)^{n}f(x).
\]
Thus, the $r$-PDB numbers can be recovered as:%
\[
\Tilde{w}_{n,r}=\left.  \left(  \frac{d}{dt}\right)  ^{n}G(t;r)\right\vert
_{t=0}=\left.  (xD)^{n}f(x)\right\vert _{x=1}.
\]

\end{proof}

Considering Lemma \ref{lem:partial_deranged_bell}, it follows
\[
\frac{1}{e}\left(  r!\tilde{w}_{n,r}-\left(  r+1\right)  !\tilde{w}
_{n,r+1}\right)  =\left.  (xD)^{n}\left(  (x-1)^{r}e^{-x}\right)  \right\vert
_{x=1}.
\]
Then from (\ref{8}), (\ref{3}) and Leibniz rule of the Mellin derivative
\cite[Eq.(3.25)]{B1}
\begin{equation}
(xD)^{n}\left(  f(x)g(x)\right)  =\sum_{j=0}^{n}{\binom{n}{j}}[(xD)^{n-j}%
f(x)][(xD)^{j}g(x)], \label{1}%
\end{equation}
we have
\begin{align*}
\frac{1}{e}\left(  r!\tilde{w}_{n,r}-\left(  r+1\right)  !\tilde{w}%
_{n,r+1}\right)   &  =\left.  (xD)^{n}\left(  (x-1)^{r}e^{-x}\right)
\right\vert _{x=1}\\
&  =\sum_{j=0}^{n}{\binom{n}{j}}\left.  (xD)^{j}\left(  (x-1)^{r}\right)
\right\vert _{x=1}\left.  (xD)^{n-j}e^{-x}\right\vert _{x=1}\\
&  =\frac{r!}{e}\sum_{j=0}^{n}{\binom{n}{j}}\text{ }{j\brace r}\tilde{\phi
}_{n-j}.
\end{align*}
Thus, we obtain a relationship between $\Tilde{w}_{n,r}$ and $\tilde{\phi}%
_{n}.$

\begin{theorem}
\label{Prop:difw}For all $n\geq r\geq0$, we have%
\begin{equation}
\tilde{w}_{n,r}-\left(  r+1\right)  \tilde{w}_{n,r+1}=\sum_{k=r}^{n}{\binom
{n}{k}}\text{ }{k\brace r}\tilde{\phi}_{n-k}. \label{9}%
\end{equation}
In particular,
\begin{equation}
\tilde{w}_{n,0}-\tilde{w}_{n,1}=\tilde{\phi}_{n}. \label{10}%
\end{equation}

\end{theorem}

\begin{remark}
When $r=1$, we use the fact that ${n \brace 1}=1$ together with
\cite[Eq.~(2.14)]{B1}%
\begin{equation}
\sum_{k=0}^{n}{\binom{n}{k}}\phi_{k}\left(  x\right)  =\phi_{n+1}\left(
x\right)  , \label{14}%
\end{equation}
for $x=-1$, we obtain
\[
\tilde{w}_{n,1}-2\tilde{w}_{n,2}=\tilde{\phi}_{n+1}-\tilde{\phi}_{n}.
\]
Combining this with (\ref{10}), we arrive at
\[
\tilde{w}_{n,0}-2\tilde{w}_{n,2}=\tilde{\phi}_{n+1}.
\]

\end{remark}

The next result presents a relation between $r$-PDB and ordered Bell numbers.

\begin{theorem}
\label{teo2}For all $n\geq r+1>0$, we have%
\[
(r+1)\Tilde{w}_{n,r+1}=\sum_{j=r}^{n-1}{\binom{n}{j}}w_{n-j}(\Tilde{w}%
_{j,r}-(r+1)\Tilde{w}_{j,r+1}).
\]

\end{theorem}

\begin{proof}
We begin by recalling Lemma \ref{lem:partial_deranged_bell}, from which we
have
\begin{align*}
\tilde{w}_{n,r}  &  =\frac{e}{r!}(xD)^{n}\left.  \left(  \frac{(x-1)^{r}%
e^{-x}}{2-x}\right)  \right\vert _{x=1}\\
&  =\frac{e}{r!}\sum_{i=0}^{r}(-1)^{r-i}{\binom{r}{i}}(xD)^{n}\left.  \left(
\frac{x^{i}e^{-x}}{2-x}\right)  \right\vert _{x=1}.
\end{align*}
Applying (\ref{1}) to the expression above allows us to write
\begin{equation}
\Tilde{w}_{n,r}=\frac{e}{r!}\sum_{i=0}^{r}(-1)^{r-i}{\binom{r}{i}}\left.
\sum_{j=0}^{n}{\binom{n}{j}}\left(  (xD)^{n-j}\left(  \frac{1}{2-x}\right)
(xD)^{j}\left(  x^{i}e^{-x}\right)  \right)  \right\vert _{x=1}. \label{11}%
\end{equation}
It is known that \cite{Dil3,Kargin3}
\begin{equation}
(xD)^{n}\left(  x^{r}e^{x}\right)  =x^{r}e^{x}\phi_{n,r}\left(  x\right)  ,
\label{13}%
\end{equation}
where $\phi_{n,r}\left(  x\right)  $ is the $n$th $r$-exponential polynomial
defined by%
\begin{equation}
\phi_{n,r}(x)=\sum_{k=0}^{n}{n+r\brace k+r}_{r}x^{k}. \label{expPol}%
\end{equation}
Here ${n+r \brace k+r}_{r}$ is the $r$-Stirling numbers of the second kind
\cite{Broder}. Thus we obtain
\[
\left.  \left[  (xD)^{j}\left(  x^{i}e^{-x}\right)  \right]  \right\vert
_{x=1}=e^{-1}\phi_{n,i}\left(  -1\right)  =e^{-1}\tilde{\phi}_{n,i},
\]
where $\tilde{\phi}_{n,r}=\phi_{n,r}\left(  -1\right)  $, which may be called
the complementary $r$-Bell numbers. Moreover, using the formula above and
(\ref{7}) in (\ref{11}), we obtain that
\[
\Tilde{w}_{n,r}=\frac{1}{r!}\sum_{i=0}^{r}(-1)^{r-i}{\binom{r}{i}}\sum
_{j=0}^{n}{\binom{n}{j}}w_{n-j}\tilde{\phi}_{j,i}.
\]
With the use of (\ref{2}) and identity \cite{Mezo1}%
\begin{equation}
\phi_{n,r}\left(  x\right)  =\sum_{k=0}^{n}r^{k}\binom{n}{k}\phi_{n-k}\left(
x\right)  , \label{15}%
\end{equation}
the equation above can be simplified as%
\begin{align*}
\Tilde{w}_{n,r}  &  =\sum_{j=0}^{n}{\binom{n}{j}}w_{n-j}\sum_{k=0}^{j}%
\binom{j}{k}{k\brace r} \tilde{\phi}_{j-k}\\
&  \overset{\text{(\ref{9})}}{=}\sum_{j=0}^{n}{\binom{n}{j}}w_{n-j}(\Tilde
{w}_{j,r}-(r+1)\Tilde{w}_{j,r+1}).
\end{align*}
Since $w_{0}=1$, the desired equation follows.
\end{proof}

The next result highlights how an infinite series involving the complementary
$r$-Bell numbers or $r$-ordered Bell numbers $w_{n,r}$, defined by \cite{Dil3}%
\[
w_{n,r}=\sum_{k=0}^{n}{n+r\brace k+r}_{r}k!,
\]
directly computes the value of $\Tilde{w}_{n,r}$, motivated by Dobinsky's
formula \cite{D} and Gross's formula \cite{Gross}.

\begin{theorem}
For any integer $n\geq1$, the $r$-PDB numbers satisfy
\[
\sum_{i=0}^{r}(-1)^{r-i}\binom{r}{i}\sum_{j=0}^{\infty}\frac{(-1)^{j}}%
{j!}w_{n,i+j}=\frac{r!}{e}\Tilde{w}_{n,r}%
\]
and
\[
\sum_{i=0}^{r}(-1)^{r-i}\binom{r}{i}\sum_{j=0}^{\infty}\frac{\tilde{\phi
}_{n,j+i}}{2^{j+1}}=r!\Tilde{w}_{n,r}.
\]
In particular, we have
\[
\sum_{j=0}^{\infty}\frac{(-1)^{j}}{j!}w_{n,j}=\frac{1}{e}\Tilde{w}_{n}\text{
and }\sum_{j=0}^{\infty}\frac{\tilde{\phi}_{n,j}}{2^{j+1}}=\Tilde{w}_{n}.
\]

\end{theorem}

\begin{proof}
From Lemma \ref{lem:partial_deranged_bell}, we have%
\begin{align*}
\tilde{w}_{n,r}  &  =\frac{e}{r!}\left.  (xD)^{n}\left(  \frac{(x-1)^{r}%
e^{-x}}{2-x}\right)  \right\vert _{x=1}\\
&  =\frac{e}{2r!}(xD)^{n}\left[  \left(  \sum_{j=0}^{\infty}(-1)^{j}%
\frac{x^{j}}{j!}\right)  \left(  \sum_{k=0}^{\infty}2^{-k}x^{k}\right)
\sum_{i=0}^{r}(-1)^{r-i}\binom{r}{i}x^{i}\right]  _{x=1}.
\end{align*}
Applying the Cauchy product, rearranging the summation order, and using the
obvious formula $\left.  (xD)^{n}x^{k+i}\right\vert _{x=1}=(k+i)^{n}$ and the
identity \cite{Dil3}
\[
\sum_{k=0}^{\infty}\frac{(k+r)^{n}}{2^{k+1}}=w_{n,r},
\]
we have
\begin{align*}
\Tilde{w}_{n,r}  &  =\frac{e}{2r!}\sum_{k=0}^{\infty}\sum_{j=0}^{k}%
\frac{(-1)^{j}}{j!}2^{j-k}\sum_{i=0}^{r}(-1)^{r-i}\binom{r}{i}(k+i)^{n}\\
&  =\frac{e}{2r!}\sum_{i=0}^{r}(-1)^{r-i}\binom{r}{i}\sum_{j=0}^{\infty}%
\frac{(-1)^{j}}{j!}\sum_{k=0}^{\infty}\frac{(k+i+j)^{n}}{2^{k}}\\
&  =\frac{e}{r!}\sum_{i=0}^{r}(-1)^{r-i}\binom{r}{i}\sum_{j=0}^{\infty}%
\frac{(-1)^{j}}{j!}w_{n,i+j}.
\end{align*}

For the proof of the second result, we apply Lemma
\ref{lem:partial_deranged_bell} to the function
\[
\frac{(x-1)^{r}e^{-x}}{(2-x)}=\sum_{i=0}^{r}(-1)^{r-i}\binom{r}{i}\sum
_{k=0}^{\infty}\frac{1}{2^{k+1}}\,x^{k+i}e^{-x},
\]
and employ (\ref{13}).
\end{proof}

\section{The partial deranged Bell polynomials \label{Sec3}}

The polynomial expansions of the Bell numbers and ordered Bell numbers, called
exponential polynomials and geometric polynomials, are defined by
\[
\phi_{n}(y)=\sum_{k=0}^{n}{n\brace k}y^{k}\text{ and }w_{n}(y)=\sum_{k=0}%
^{n}{n\brace k}k!y^{k},
\]
respectively. Please see
\cite{AAR,AN,BD1,B,BoyadzhievandDil,Dil2,Kargin1,Kargin2,Mihoubi,NBCC} for
more details. Inspired by the polynomial expansions of the Bell numbers and
ordered Bell numbers, we introduce a polynomial expansion for $r$-PDB numbers.
By deriving the corresponding exponential generating function, we investigate
various properties of the associated polynomials. Using these properties, we
demonstrate that certain finite sums involving Stirling numbers of the second
kind, Bernoulli numbers, and binomial coefficients can be computed in terms of
the deranged numbers.

Let us define $r$-PDB polynomials as follows:
\begin{equation}
\Tilde{w}_{n,r}(y)=\sum_{k=0}^{n}{n\brace k}d_{k,r}y^{k}. \label{eq:B}%
\end{equation}
It is obvious that $\Tilde{w}_{n,r}(1)=\Tilde{w}_{n,r}$ and $\Tilde{w}%
_{n,0}=\Tilde{w}_{n}.$ Thus, $\Tilde{w}_{n,0}(y)=\Tilde{w}_{n}(y)$, may called
\emph{deranged Bell polynomials. }

It is straightforward to show that the generating function of $\Tilde{w}%
_{n,r}(y)$ is
\begin{equation}
\sum_{n=0}^{\infty}\Tilde{w}_{n,r}(y)\frac{t^{n}}{n!}=\frac{\left(
y(e^{t}-1)\right)  ^{r}}{r!}\frac{e^{-y(e^{t}-1)}}{1-y(e^{t}-1)}. \label{eq:A}%
\end{equation}
Using this generating function, we present several properties of the $r$-PDB polynomials.

The first result is a recurrence relation:

\begin{theorem}
For any non-negative integers $n$, $m$, and $r$ with $n\geq m+r$, we have
\begin{equation}
\Tilde{w}_{n,m+r}(y)=\frac{y^{r}}{{\binom{m+r}{m}}}\sum_{k=0}^{n}{\binom{n}%
{k}}{n-k\brace r}\Tilde{w}_{k,m}(y). \label{eq:C}%
\end{equation}

\end{theorem}

\begin{proof}
Using the generating function \eqref{eq:A}, we write
\[%
\begin{split}
\sum_{n=0}^{\infty}\Tilde{w}_{n,m+r}(y)\frac{t^{n}}{n!}  &  =\frac{\left(
y(e^{t}-1)\right)  ^{m+r}}{(m+r)!}\frac{e^{-y(e^{t}-1)}}{1-y(e^{t}-1)}\\
&  =\frac{y^{r}}{{\binom{m+r}{m}}}\sum_{n=0}^{\infty}\left[  \sum_{k=0}%
^{n}{\binom{n}{k}}{n-k\brace r}\Tilde{w}_{k,m}(y)\right]  \frac{t^{n}}{n!}.
\end{split}
\]
Comparing the coefficients of $\frac{t^{n}}{n!}$ on both sides gives \eqref{eq:C}.
\end{proof}

Using \eqref{eq:B} in \eqref{eq:C} gives%
\[
\tilde{w}_{n,r+m}(y)=\frac{y^{r}}{{\binom{m+r}{m}}}\sum_{j=0}^{n}\sum
_{k=j}^{n}{\binom{n}{k}}{n-k\brace m}{k\brace j}d_{j,r}y^{j}.
\]
On the other hand, from \eqref{eq:B}, the left-hand side can be written as
\[
\Tilde{w}_{n,r+m}(y)=\sum_{j=0}^{n}{n\brace j}d_{j,r+m}y^{j}.
\]
Thus, these two representations yield the following corollary.

\begin{corollary}
For suitable nonnegative integers $j,m,n$ and $r$, we have
\[
\sum_{k=j-r}^{n}{\binom{n}{k}}{n-k\brace m}{k\brace j-r}={\binom{m+r}{m}%
}{n\brace j}\frac{d_{j,r+m}}{d_{j-r,r}}.
\]

\end{corollary}

\subsection{$r$-PDB polynomials via exponential polynomials}

In this subsection, we demonstrate that the $r$-PDB polynomials can be
represented in terms of exponential polynomials, thereby extending the scope
of the Theorem \ref{Prop:difw}. This representation reveals that a finite sum
involving binomial coefficients and Stirling numbers of the second kind can
also be expressed in terms of $r$-PDB numbers (cf. Corollary \ref{cor5} below).

\begin{theorem}
Let $n$ and $r$ be nonnegative integers. Then we have
\begin{equation}
\Tilde{w}_{n,r}(y)-\left(  r+1\right)  \Tilde{w}_{n,r+1}(y)=y^{r}\sum
_{k=r}^{n}{\binom{n}{k}}{k\brace r}\phi_{n-k}(-y). \label{eq:D}%
\end{equation}

\end{theorem}

\begin{proof}
The exponential generating function \eqref{eq:A} can be rewritten in the form
\[%
\begin{split}
r!\sum_{n=0}^{\infty}\Tilde{w}_{n,r}(y)\frac{t^{n}}{n!}  &  =\left(
-1+\frac{1}{1-y(e^{t}-1)}\right)  (r-1)!\frac{\left(  y(e^{t}-1)\right)
^{r-1}e^{-y(e^{t}-1)}}{(r-1)!}\\
&  =-(r-1)!y^{r-1}\frac{\left(  e^{t}-1\right)  ^{r-1}}{(r-1)!}e^{-y(e^{t}%
-1)}\\
&  \quad+(r-1)!\frac{\left(  y(e^{t}-1)\right)  ^{r-1}}{(r-1)!}\frac
{e^{-y(e^{t}-1)}}{1-y(e^{t}-1)}.
\end{split}
\]
Then, by using \eqref{eq:A} together with (\ref{s2gf}), we have
\[%
\begin{split}
r!\sum_{n=0}^{\infty}\Tilde{w}_{n,r}(y)\frac{t^{n}}{n!}  &  =-(r-1)!y^{r-1}%
\left(  \sum_{n=0}^{\infty}{n\brace r-1}\frac{t^{n}}{n!}\right)  \left(
\sum_{n=0}^{\infty}\phi_{n}(-y)\frac{t^{n}}{n!}\right) \\
&  \quad+(r-1)!\sum_{n=0}^{\infty}\Tilde{w}_{n,r-1}(y)\frac{t^{n}}{n!}.
\end{split}
\]
Using the Cauchy product of the series and comparing the coefficients on both
sides yields the desired formula \eqref{eq:D}.
\end{proof}

The polynomial $\Tilde{w}_{n,r}(y)$ can also be written in terms of
$r$-exponential polynomials $\phi_{n,r}(x)$. This is summarized in the
following corollary, which follows from (\ref{2}) and (\ref{15}).

\begin{corollary}
The sequence $\Tilde{w}_{n,r}(y)$ satisfies the following equation:
\begin{equation}
\Tilde{w}_{n,r}(y)-\left(  r+1\right)  \Tilde{w}_{n,r+1}(y)=\frac{y^{r}}%
{r!}\sum_{i=0}^{r}(-1)^{r-i}{\binom{r}{i}}\phi_{n,i}(-y). \label{eq:E}%
\end{equation}

\end{corollary}

The first result of the following corollary follows from equations
\eqref{eq:B}, \eqref{eq:D}, and (\ref{expPol}). The second formula follows
from equation \eqref{eq:E} and (\ref{expPol}).

\begin{corollary}
\label{cor5}For all $j\geq r-1$, we have:
\[
\sum_{k=j-r+1}^{n}{\binom{n}{k}}{n-k\brace r-1}{k\brace j-r+1}=(-1)^{j-r+1}%
{n\brace j}(d_{j,r-1}-rd_{j,r})
\]
and
\[
\sum_{i=0}^{r-1}(-1)^{i}{\binom{r-1}{i}}{n+i\brace j-(r-1-i)}_{i}%
=(-1)^{j}{n\brace j}(d_{j,r-1}-rd_{j,r}).
\]

\end{corollary}

\subsection{$r$-PDB polynomials via geometric polynomials}

In this subsection, we investigate the structural relationship between the
$r$-PDB and geometric polynomials, highlighting how the properties of the
geometric polynomials can be utilized to derive identities and expressions
involving the $r$-PDB polynomials.

From (\ref{eq:A}) one can see that
\begin{equation}
\sum_{n=0}^{\infty}\sum_{r=0}^{\infty}\Tilde{w}_{n,r}\left(  y\right)
z^{r}\frac{t^{n}}{n!}=\frac{e^{\left(  z-1\right)  y(e^{t}-1)}}{(1-y\left(
e^{t}-1\right)  )}, \label{dgf}%
\end{equation}
which can be written as%
\begin{equation}
\sum_{n=0}^{\infty}\sum_{r=0}^{\infty}\Tilde{w}_{n,r}\left(  y\right)
z^{r}\frac{t^{n}}{n!}=e^{zy(e^{t}-1)}\frac{e^{-y(e^{z}-1)}}{(1-y\left(
e^{z}-1\right)  )}. \label{dgf1}%
\end{equation}
It is also known that the geometric and exponential polynomials have the
generating functions \cite{B}
\[
\sum_{n=0}^{\infty}w_{n}(x)\frac{t^{n}}{n!}=\frac{1}{1-x(e^{t}-1)}\text{ and
}\sum_{n=0}^{\infty}\phi_{n}(x)\frac{t^{n}}{n!}=e^{x(e^{t}-1)},
\]
respectively. Thus (\ref{dgf}) and (\ref{dgf1}) give rise to a relationship
between the $r$-PDB, geometric, exponential, and deranged Bell polynomials:

\begin{proposition}
For the sequence $(\Tilde{w}_{n,r}\left(  y\right)  )_{n\geq0}$, we have
\begin{equation}
\sum_{r=0}^{n}\Tilde{w}_{n,r}\left(  y\right)  z^{r}=\sum_{r=0}^{n}\binom
{n}{r}\phi_{r}(\left(  z-1\right)  y)w_{n-r}\left(  y\right)  \label{5}%
\end{equation}
and
\[
\sum_{r=0}^{n}\Tilde{w}_{n,r}\left(  y\right)  z^{r}=\sum_{r=0}^{n}\binom
{n}{r}\phi_{r}(zy)\tilde{w}_{n-r}\left(  y\right)  .
\]

\end{proposition}

This proposition leads to two distinct results. The first one is presented below:

\begin{corollary}
For non-negative integers $n$,
\begin{equation}
\sum_{r=0}^{n}\Tilde{w}_{n,r}\left(  y\right)  =w_{n}\left(  y\right)
\label{fw}%
\end{equation}
and%
\[
\sum_{r=0}^{n}\binom{n}{r}\phi_{r}(y)\tilde{w}_{n-r}\left(  y\right)
=w_{n}\left(  y\right)  .
\]

\end{corollary}

\begin{remark}
The validity of identity (\ref{fw}) for $y=1$ can be demonstrated by
recognizing that the computation of the number of all permutations of the
blocks can be achieved through $w_{n}$ (RHS). Alternatively, we can interpret
this as counting the permutations of the blocks with $0$ fixed points, $1$
fixed point, and so on, which is represented by $\sum_{r=0}^{n}\Tilde{w}%
_{n,r}$ (LHS).

For $z=0$ in (\ref{5}), we have%
\[
\sum_{r=0}^{n}\binom{n}{r}\phi_{r}(-y)w_{n-r}\left(  y\right)  =\Tilde{w}%
_{n}\left(  y\right)  ,
\]
Since $w_{n}\left(  -1\right)  =\left(  -1\right)  ^{n}$, for $y=-1$, we reach
at%
\[
\sum_{r=0}^{n}\left(  -1\right)  ^{n-r}\binom{n}{r}\phi_{r}=\Tilde{w}%
_{n}\left(  -1\right)  ,
\]
which is an alternating version of (\ref{14}) for $x=1$.
\end{remark}

The second one is given as follows:

\begin{corollary}
For all non-negative integer $n$,%
\begin{equation}
\sum_{r=0}^{n}\left(  -1\right)  ^{r}d_{r}\Tilde{w}_{n,r}\left(  y\right)
=\frac{w_{n}\left(  y\right)  +w_{n}\left(  -y\right)  }{2}. \label{12}%
\end{equation}
In particular,
\[
\sum_{r=0}^{n}\left(  -1\right)  ^{r}d_{r}\Tilde{w}_{n,r}=\sum_{r=0}%
^{n}\left(  -1\right)  ^{r}d_{r}\Tilde{w}_{n,r}\left(  -1\right)  =\frac
{w_{n}+\left(  -1\right)  ^{n}}{2}.
\]

\end{corollary}

\begin{proof}
Setting $z\rightarrow1-z$ in (\ref{5}) gives:
\[
\sum_{r=0}^{n}\Tilde{w}_{n,r}\left(  y\right)  \sum_{k=0}^{r}\binom{r}%
{k}\left(  -1\right)  ^{k}z^{k}=\sum_{k=0}^{n}\binom{n}{k}\phi_{k}%
(-zy)w_{n-k}\left(  y\right)  .
\]
Multiplying both sides with $e^{-z}$, integrating from zero to infinity with
respect to $z$, and using the relations \cite{B}
\[
\int\limits_{0}^{\infty}\phi_{k}(zy)e^{-z}dz=w_{k}\left(  y\right)  \text{ and
}\int\limits_{0}^{\infty}\lambda^{k}e^{-\lambda}d\lambda=k!,
\]
we have
\[
\sum_{r=0}^{n}\left(  -1\right)  ^{r}d_{r}\Tilde{w}_{n,r}\left(  y\right)
=\sum_{r=0}^{n}\binom{n}{r}w_{n-r}\left(  y\right)  w_{r}\left(  -y\right)  .
\]
From \cite{Kargin1}, we know that
\[
\sum_{r=0}^{n}\binom{n}{r}w_{n-r}\left(  y\right)  w_{r}\left(  -y\right)
=\frac{w_{n}\left(  y\right)  +w_{n}\left(  -y\right)  }{2},
\]
the desired formula follows.
\end{proof}

It is good to note that comparing the coefficients of variable $y$ in
(\ref{12}) gives the well-known formula for derangement numbers
\[
\sum_{r=0}^{k}\left(  -1\right)  ^{r}\binom{k}{r}d_{r}d_{k-r}=\left\{
\begin{array}
[c]{cc}%
0 & \text{if }k\text{ is odd},\\
k! & \text{if }k\text{ is even}.
\end{array}
\right.
\]

Summing both sides of (\ref{eq:D}) with respect to $r$ from zero to $n$, we
have
\[
\sum_{r=0}^{n}\Tilde{w}_{n,r}(y)-\sum_{r=0}^{n}\left(  r+1\right)  \Tilde
{w}_{n,r+1}(y)=\sum_{k=0}^{n}{\binom{n}{k}}\phi_{n-k}(-y)\phi_{k}(y).
\]
Using (\ref{fw}) and the fact that the right-hand side of the previous
equation is zero \cite{B1}, we obtain the following corollary.

\begin{corollary}
For any positive integer $n$, we have%
\[
\sum_{r=1}^{n}r\Tilde{w}_{n,r}(y)=w_{n}(y).
\]

\end{corollary}

\subsection{$r$-PDB polynomials via Bernoulli numbers}

This subsection is devoted to exploring the relationship between $r$-PDB
polynomials and higher-order (generalized) Bernoulli numbers $B_{n}^{(r)}$,
defined by \cite[Chapter 6]{N}%
\begin{equation}
\sum_{n=0}^{\infty}B_{n}^{(r)}\frac{t^{n}}{n!}=\left(  \frac{t}{e^{t}%
-1}\right)  ^{r},~\left\vert t\right\vert <2\pi. \label{16}%
\end{equation}

\begin{theorem}
For all positive integers $n,m,r$ with $n\geq m$, we have
\begin{equation}
\sum_{k=m}^{n}{\binom{n+r}{k+r}}B_{n-k}^{(r)}\Tilde{w}_{k+r,m+r}%
(y)=\frac{{\binom{n+r}{r}}}{{\binom{m+r}{m}}}y^{r}\Tilde{w}_{n,m}(y).
\label{6}%
\end{equation}
When $r=1$,
\[
\sum_{k=m}^{n}{\binom{n}{k}}B_{n-k}\Tilde{w}_{k,m}(y)=\frac{n}{m}y\Tilde
{w}_{n-1,m-1}(y),
\]
where $B_{n}=B_{n}^{(1)}$ are the classical Bernoulli numbers.
\end{theorem}

\begin{proof}
From \eqref{eq:A}, we have
\[
\frac{1}{\left(  e^{t}-1\right)  ^{r}}\sum_{n=0}^{\infty}\Tilde{w}%
_{n,r+m}(y)\frac{t^{n}}{n!}=\frac{m!\,y^{r}}{(r+m)!\,}\frac{\left(  y\left(
e^{t}-1\right)  \right)  ^{m}e^{-y\left(  e^{t}-1\right)  }}{m!\left(
1-y\left(  e^{t}-1\right)  \right)  }.
\]
Using (\ref{16}), we obtain
\begin{align*}
&  \sum_{n=0}^{\infty}B_{n}^{(r)}\frac{t^{n}}{n!}\sum_{n=0}^{\infty}\Tilde
{w}_{n,r+m}(y)\frac{t^{n}}{n!}=\frac{m!\,y^{r}}{(r+m)!}\sum_{n=0}^{\infty
}\Tilde{w}_{n,m}(y)\frac{t^{n+r}}{n!},\\[1mm]
&  \sum_{n=0}^{\infty}\left[  \sum_{k=0}^{n}{\binom{n}{k}}\Tilde{w}%
_{n-k,r+m}(y)B_{k}^{(r)}\right]  \frac{t^{n}}{n!}=\frac{y^{r}}{{\binom{m+r}%
{m}}}\sum_{n=r}^{\infty}\binom{n}{r}\Tilde{w}_{n-r,m}(y)\frac{t^{n}}{n!}.
\end{align*}
Comparing the coefficients of $\frac{t^{n}}{n!}$ on both sides yields the
desired formula.
\end{proof}

Using (\ref{eq:B}) in (\ref{6}) gives the following explicit formula:

\begin{corollary}
For appropriate integers $j,m,n$, and $r$ we have
\[
\sum_{k=j}^{n}{\binom{n+r}{k+r}}{k+r\brace j+r}B_{n-k}^{(r)}=\frac
{{\binom{n+r}{r}}}{{\binom{j+r}{r}}}{n\brace j},
\]
In particular%
\[
\sum_{k=j}^{n}{\binom{n+1}{k+1}}{k+1\brace j+1}B_{n-k}=\frac{n+1}%
{j+1}{n\brace
j}.
\]

\end{corollary}

\end{document}